\documentclass{amsart}

\usepackage[T1]{fontenc}
\usepackage[latin1]{inputenc}
\usepackage{url}
\usepackage{pstricks}
\usepackage{dsfont}
\usepackage{enumitem}
\usepackage[colorlinks=true, linktoc=page]{hyperref}

\makeatletter

\def\part{\@startsection{part}{1}%
\z@{.7\linespacing\@plus\linespacing}{.5\linespacing}%
{\large\normalfont\bfseries}}

\def\section{\@startsection{section}{1}%
\z@{.7\linespacing\@plus\linespacing}{.5\linespacing}%
{\normalfont\bfseries\centering}}

\def\@settitle{\begin{center}%
  \baselineskip14\p@\relax
    \bfseries
    \LARGE\@title
  \end{center}%
}

\def\@setauthors{%
  \begingroup
  \trivlist
  \centering\footnotesize \@topsep30\p@\relax
  \advance\@topsep by -\baselineskip
  \item\relax
  \andify\authors
  \def\\{\protect\linebreak}%
 {\Large\authors}%
  \endtrivlist
  \endgroup
}

\def\maketitle{\par
  \@topnum\z@ 
  \@setcopyright
  \thispagestyle{firstpage}
  \ifx\@empty\shortauthors \let\shortauthors\shorttitle
  \else \andify\shortauthors
  \fi
  \@maketitle@hook
  \begingroup
  \@maketitle
  \toks@\@xp{\shortauthors}\@temptokena\@xp{\shorttitle}%
  \toks4{\def\\{ \ignorespaces}}
  \edef\@tempa{%
    \@nx\markboth{\the\toks4
      \@nx{\the\toks@}}{\the\@temptokena}}%
  \@tempa
  \endgroup
  \c@footnote\z@
  \def\do##1{\let##1\relax}%
  \do\maketitle \do\@maketitle \do\title \do\@xtitle \do\@title
  \do\author \do\@xauthor \do\address \do\@xaddress
  \do\email \do\@xemail \do\curraddr \do\@xcurraddr
  \do\commby \do\@commby
  \do\dedicatory \do\@dedicatory \do\thanks \do\thankses
  \do\keywords \do\@keywords \do\subjclass \do\@subjclass
}

\makeatother

\newcommand{\dd}{\mathrm{d}}
\newcommand{\R}{\mathbb{R}}
\newcommand{\Exp}{\mathbb{E}}
\renewcommand{\Pr}{\mathbb{P}}
\newcommand{\ind}[1]{\mathds{1}_{\{#1\}}}

\newcommand{\uto}{\uparrow}
\newcommand{\dto}{\downarrow}

\newcommand{\Cs}{\mathrm{C}}
\newcommand{\Ws}{\mathrm{W}}
\newcommand{\Neigh}{\mathcal{V}}

\newcommand{\sk}{\vskip 3mm}

\newtheorem{defi}{Definition}
\newtheorem{lem}[defi]{Lemma}
\newtheorem{prop}[defi]{Proposition}
\newtheorem{theo}[defi]{Theorem}
\newtheorem{cor}[defi]{Corollary}

\newtheoremstyle{myremark}{}{}{}{0pt}{\bfseries}{.}{ }{}
\theoremstyle{myremark}
\newtheorem{rk}[defi]{Remark}

\title{Chaoticity of the stationary distribution of rank-based interacting diffusions}

\author{Julien Reygner}
\address{{\rm \noindent Sorbonne Universités, UPMC Univ Paris 06, UMR 7599, LPMA, F-75005 Paris.\newline
\indent Université Paris-Est, CERMICS (ENPC), F-77455 Marne-la-Vallée.}}
\curraddr{{\rm Laboratoire de Physique, \'Ecole Normale Sup\'erieure de Lyon, 46 all\'ee d'Italie, F-69364 Lyon.}}
\email{\href{mailto:julien.reygner@polytechnique.org}{julien.reygner@polytechnique.org}}


\begin{document}

\begin{abstract}
  The mean-field limit of systems of rank-based interacting diffusions is known to be described by a nonlinear diffusion process. We obtain a similar description at the level of stationary distributions. Our proof is based on explicit expressions for the Laplace transforms of these stationary distributions and yields convergence of the marginal distributions in Wasserstein distances of all orders. We highlight the consequences of this result on the study of rank-based models of equity markets, such as the Atlas model. 
\end{abstract}

\maketitle

\section{Introduction}

\subsection{Rank-based interacting diffusions} Let $b : [0,1] \to \R$ be a continuous function, and $\sigma \not= 0$. For all $n \geq 1$, consider the system of rank-based interacting diffusions, or particles,
\begin{equation}\label{eq:SDE}
  \forall i \in \{1, \ldots, n\}, \qquad \dd X^n_i(t) = b_n\left(\sum_{j=1}^n \ind{X^n_j(t) \leq X^n_i(t)}\right) \dd t + \sigma \dd W_i(t),
\end{equation}
where $(W_1(t), \ldots, W_n(t))_{t \geq 0}$ is a standard Brownian motion in $\R^n$, and for all $k \in \{1, \ldots, n\}$,
\begin{equation*}
  b_n(k) := n \int_{v=\frac{k-1}{n}}^{\frac{k}{n}} b(v)\dd v = n\left(B\left(\frac{k}{n}\right) - B\left(\frac{k-1}{n}\right)\right), \qquad B(u) := \int_{v=0}^u b(v)\dd v.
\end{equation*}
By the Girsanov theorem, the stochastic differential equation~\eqref{eq:SDE} possesses a unique weak solution, and actually a unique strong solution, see~\cite{veretennikov} for instance. When the number $n$ of particles grows to infinity, propagation of chaos results toward the unique weak solution to the nonlinear (in McKean's sense) stochastic differential equation
\begin{equation}\label{eq:NLSDE}
  \left\{\begin{aligned}
    & \dd X(t) = b(F_t(X(t))) \dd t + \sigma \dd W(t),\\
    & F_t(x) = \Pr(X(t) \leq x),
  \end{aligned}\right.
\end{equation}
where obtained in~\cite{jourdain:signed, jm}; we also refer to~\cite{jourey, dsvz} for nonconstant diffusion coefficients. 

These propagation of chaos results are not uniform in time, and therefore do not provide any indication on the link between the long time behaviour of the particle system, which was studied in~\cite{pp,jm}, and the long time behaviour of the nonlinear diffusion process, which was described in~\cite{jm,jourey}. The purpose of this note is to clarify this link by showing that the stationary distribution of a suitably modified version of the particle system is chaotic with respect to the stationary distribution of the nonlinear process.

\subsection{Projected particle system and nonlinear diffusion process}\label{ss:intro:stat} As was remarked in~\cite{jm}, the solution to~\eqref{eq:SDE} cannot converge to an equilibrium, since its projection along the direction $(1, \ldots, 1)$ is a drifted Brownian motion. One can however address the long time behaviour of the projection onto the hyperplane
\begin{equation*}
  M_n := \{(z_1, \ldots, z_n) \in \R^n : z_1 + \cdots + z_n = 0\},
\end{equation*}
which is orthogonal to the singular direction $(1, \ldots, 1)$. The resulting process is called the {\em projected particle system}, it is the $M_n$-valued diffusion process solving
\begin{equation*}
    \dd Z^n_i(t) = \left(b_n\left(\sum_{j=1}^n \ind{Z^n_j(t) \leq Z^n_i(t)}\right)-\bar{b}\right)\dd t + \sigma\frac{n-1}{n} \dd W_i(t) - \frac{\sigma}{n} \sum_{j \not= i} \dd W_j(t),
\end{equation*}
where $\bar{b} := \frac{1}{n} \sum_{k=1}^n b_n(k) = B(1)$. Propagation of chaos for the projected particle system toward the nonlinear diffusion process~\eqref{eq:NLSDE} was established in~\cite{jm}. The long time behaviour of the projected particle system was addressed by~\cite{pp,jm}, under the following natural assumption:
\begin{enumerate}[label=(E), ref=E]
  \item \label{ass:E} $\bar{b}=B(1)=0$, and the function $b$ is decreasing on $[0,1]$.
\end{enumerate}
The following proposition is due to Pal and Pitman~\cite[Theorem~8]{pp}; see also Jourdain and Malrieu~\cite[Theorem~2.12]{jm}. We use the notation $z_{(1)} \leq \cdots \leq z_{(n)}$ to refer to the order statistics of a vector $(z_1, \ldots, z_n) \in \R^n$.
\begin{prop}\label{prop:pp}
  Under Assumption~\eqref{ass:E}, for all $n \geq 1$,
  \begin{equation*}
    \mathcal{Z}_n := \int_{z \in M_n} \exp\left(\frac{2}{\sigma^2} \sum_{k=1}^n b_n(k) z_{(k)}\right) \dd z < +\infty,
  \end{equation*}
  and the probability distribution with density
  \begin{equation*}
    p^n_{\infty}(z) := \frac{1}{\mathcal{Z}_n} \exp\left(\frac{2}{\sigma^2} \sum_{k=1}^n b_n(k) z_{(k)}\right)
  \end{equation*}
  with respect to the surface measure $\dd z$ on $M_n$ is the unique stationary distribution of the process $(Z^n_1(t), \ldots, Z^n_n(t))_{t \geq 0}$.
\end{prop}
Let us remark that the density $p^n_{\infty}(z)$ only depends on the order statistics of $z$, and therefore is invariant under the permutations of the coordinates of $z$. As a consequence, the probability distribution $P^n_{\infty} := p^n_{\infty}(z)\dd z$ is a symmetric probability distribution on $\R^n$, which gives full measure to $M_n$.

\sk
On the other hand, the stationary distributions of the nonlinear diffusion process were described in~\cite{jourey}. This description relies on the function $\Phi$ introduced in Lemma~\ref{lem:Psi} below.

\begin{lem}\label{lem:Psi}
  Under Assumption~\eqref{ass:E},
  \begin{enumerate}[label=(\roman*), ref=\roman*]
    \item $b(0)>0>b(1)$ and, for all $u \in (0,1)$, $B(u)>0$,
    \item the function $\Phi : (0,1) \to \R$ defined by 
    \begin{equation}\label{eq:Psialt}
      \forall u \in (0,1), \qquad \Phi(u) := \int_{v=0}^u \frac{v \sigma^2}{2B(v)}\dd v - \int_{v=u}^1 \frac{(1-v)\sigma^2}{2B(v)}\dd v
    \end{equation}
    is $\Cs^2$ and increasing on $(0,1)$, and satisfies
    \begin{equation}\label{eq:Psiequiv}
      \Phi(u) \sim \frac{\sigma^2}{2b(0)} \log (u) \quad \text{when $u \dto 0$}, \qquad \Phi(u) \sim \frac{\sigma^2}{2b(1)} \log (1-u) \quad \text{when $u \uto 1$}.
    \end{equation}
    Besides, it is integrable on $[0,1]$ and such that
    \begin{equation}\label{eq:Psicentre}
      \int_{u=0}^1 \Phi(u)\dd u = 0.
    \end{equation}
  \end{enumerate}
\end{lem}
\begin{proof}
  Under Assumption~\eqref{ass:E}, the chain of inequalities $b(0)>0>b(1)$ is straightforward. Besides, for all $u \in (0,1)$,
  \begin{equation*}
    \frac{1}{u} \int_{v=0}^u b(v)\dd v > b(u) > \frac{1}{1-u} \int_{v=u}^1 b(v)\dd v,
  \end{equation*}
  which implies $(1-u)B(u) > u(B(1)-B(u))$ and finally $B(u) > uB(1) = 0$, whence the first point.
  
  Assumption~\eqref{ass:E} combined with the continuity of $b$ also implies that
  \begin{itemize}
    \item when $u \dto 0$, $B(u) \sim b(0)u$, with $b(0)>0$,
    \item when $u \uto 1$, $B(u) \sim -b(1)(1-u)$, with $b(1)<0$,
  \end{itemize}
  therefore the integrals in the right-hand side of~\eqref{eq:Psialt} are finite, and the function $\Phi$ is $\Cs^2$ and increasing on $(0,1)$, and satisfies~\eqref{eq:Psiequiv}. The integrability of $\Phi$ on $[0,1]$ follows from~\eqref{eq:Psiequiv} and the continuity of $\Phi$ on $(0,1)$, and by the Fubini-Tonelli theorem,
  \begin{equation*}
    \int_{u=0}^1 \int_{v=0}^u \frac{v \sigma^2}{2B(v)}\dd v\dd u = \int_{v=0}^1 \frac{v(1-v) \sigma^2}{2B(v)}\dd v = \int_{u=0}^1 \int_{v=u}^1 \frac{(1-v) \sigma^2}{2B(v)}\dd v\dd u,
  \end{equation*}
  whence~\eqref{eq:Psicentre}.
\end{proof}

Note that the inverse function $\Phi^{-1}$ of the function $\Phi$ defined in Lemma~\ref{lem:Psi} is the cumulative distribution function $F_{\infty}$ of a probability distribution $P_{\infty}$ on $\R$, which is such that
\begin{equation*}
  \int_{x \in \R} |x|P_{\infty}(\dd x) = \int_{u=0}^1 |\Phi(u)|\dd u < +\infty, \qquad \int_{x \in \R} xP_{\infty}(\dd x) = \int_{u=0}^1 \Phi(u)\dd u = 0.
\end{equation*}
Besides, since $\Phi$ is $\Cs^2$ and $\Phi'(u) > 0$ for all $u \in (0,1)$, we deduce that $P_{\infty}$ possesses a density $p_{\infty}$ with respect to the Lebesgue measure on $\R$, which writes $p_{\infty}(x) = \frac{2}{\sigma^2} B(F_{\infty}(x))$. 

We can now recall the description of the set of stationary distributions of the nonlinear process, which follows from~\cite[Proposition~4.1]{jourey}.

\begin{prop}\label{prop:jourey}
  Under Assumption~\eqref{ass:E}, the stationary probability distributions for the nonlinear process $(X_t)_{t \geq 0}$ are the translations of the probability distribution $P_{\infty}$; that is to say, the probability distributions with cumulative distribution function $x \mapsto F_{\infty}(x+\bar{x})$ for some $\bar{x} \in \R$.
\end{prop}

Ergodicity results for the nonlinear diffusion process were obtained in~\cite{jm, jourey}. Let us precise that in~\cite{jourey}, the stationary distributions are proven to be the translations of the function $\Psi$ defined on $(0,1)$ by
\begin{equation*}
  \forall u \in (0,1), \qquad \Psi(u) := \int_{v=\frac{1}{2}}^u \frac{\sigma^2}{2B(v)}\dd v.
\end{equation*}
Since $\Phi$ and $\Psi$ have the same derivative, it is clear that the set of translations of $\Phi^{-1}$ coincides with the set of translations of $\Psi^{-1}$.

As a consequence of Proposition~\ref{prop:jourey}, a stationary distribution for the nonlinear process is characterised by its expectation. In particular, $P_{\infty}$ is the unique centered stationary distribution of the nonlinear process.

\subsection{Main results and outline} We are now ready to state our main results and detail the outline of the paper.

Let us first recall the definition of the notion of chaoticity~\cite[Definition~2.1, p.~177]{sznitman}. If $P^n$ is a probability distribution on $\R^n$ and $k \in \{1, \ldots, n\}$, we denote by $P^{k,n}$ the marginal distribution of the $k$ first coordinates under $P^n$.
\begin{defi}\label{defi:chaoticity}
  For all $n \geq 1$, let $P^n$ be a symmetric probability distribution on $\R^n$, and let $P$ be a probability distribution on $\R$. The sequence $(P^n)_{n \geq 1}$ is said to be $P$-chaotic if, for all $k \geq 1$, $P^{k,n}$ converges weakly to the product measure $P^{\otimes k}$.
\end{defi}

Recall that we denote by $P^n_{\infty}$ the unique stationary distribution of the projected particle system; it is the probability distribution on $\R^n$ with density $p^n_{\infty}(z)$ with respect to the surface measure $\dd z$ on $M_n$. On the other hand, $P_{\infty}$ refers to the unique centered stationary distribution of the nonlinear diffusion process; it is the probability distribution with density $p_{\infty}$ with respect to the Lebesgue measure on $\R$. Of course, our purpose is to establish the $P_{\infty}$-chaoticity of the sequence $(P^n_{\infty})_{n \geq 1}$. Our proof is based on the study of the Laplace transform
\begin{equation*}
  L^{2,n}_{\infty}(s,t) := \int_{z \in M_n} \exp(sz_1 + tz_2) p^n_{\infty}(z)\dd z
\end{equation*}
of $P^{2,n}_{\infty}$, and of the Laplace transform
\begin{equation*}
  L_{\infty}(r) := \int_{x \in \R} \exp(rx) p_{\infty}(x) \dd x
\end{equation*}
of $P_{\infty}$. Following the results of Subsection~\ref{ss:intro:stat}, we can already obtain an explicit expression of $L_{\infty}(r)$; indeed, since the inverse of the cumulative distribution function $F_{\infty}$ of $P_{\infty}$ is $\Phi$, then, for all $r \in \R$,
\begin{equation*}
  L_{\infty}(r) = \int_{u=0}^1 \exp(r\Phi(u)) \dd u.
\end{equation*} 
Besides, under Assumption~\eqref{ass:E}, the point~\eqref{eq:Psiequiv} in Lemma~\ref{lem:Psi} ensures that as soon as $r$ is taken in the set
\begin{equation*}
  \Neigh := \{r \in \R : -2b(0)/\sigma^2 < r < -2b(1)/\sigma^2\},
\end{equation*}
then $L_{\infty}(r) < +\infty$. This is the first part of Theorem~\ref{theo:main}.

\begin{theo}\label{theo:main}
  Under Assumption~\eqref{ass:E},
  \begin{enumerate}[label=(\roman*), ref=\roman*]
    \item\label{it:main:1} for all $r \in \Neigh$, $L_{\infty}(r)$ is finite and writes
    \begin{equation*}
      L_{\infty}(r) = \int_{u=0}^1 \exp(r\Phi(u)) \dd u;
    \end{equation*}
    \item\label{it:main:2} for all $(s,t)$ taken in the set
    \begin{equation*}
      \Neigh_2 := \{(s,t) \in \Neigh \times \Neigh : s+t \in \Neigh\},
    \end{equation*}
    there exists $n_0 \geq 2$ such that, for all $n \geq n_0$,  $L^{2,n}_{\infty}(s,t)$ is finite and writes
    \begin{equation*}
      L^{2,n}_{\infty}(s,t) = \frac{1}{n(n-1)} \sum_{i=1}^n \sum_{j \not= i} J^n_{i,j}(s,t),
    \end{equation*}
    where, for all $1 \leq i < j \leq n$,
    \begin{equation*}
      J^n_{i,j}(s,t) := \prod_{k=1}^{i-1} \frac{1}{1 - (s+t)\frac{\sigma^2}{2n} \frac{k/n}{B(k/n)}}\prod_{k=i}^{j-1} \frac{1}{1 - t\frac{\sigma^2}{2n} \frac{k/n}{B(k/n)} + s\frac{\sigma^2}{2n} \frac{1-k/n}{B(k/n)}}\prod_{k=j}^{n-1} \frac{1}{1 + (s+t)\frac{\sigma^2}{2n} \frac{1-k/n}{B(k/n)}},
    \end{equation*}
    while, for all $1 \leq j < i \leq n$,
    \begin{equation*}
      J^n_{i,j}(s,t) := \prod_{k=1}^{j-1} \frac{1}{1 - (s+t)\frac{\sigma^2}{2n} \frac{k/n}{B(k/n)}}\prod_{k=j}^{i-1} \frac{1}{1 - s\frac{\sigma^2}{2n} \frac{k/n}{B(k/n)} + t\frac{\sigma^2}{2n} \frac{1-k/n}{B(k/n)}}\prod_{k=i}^{n-1} \frac{1}{1 + (s+t)\frac{\sigma^2}{2n} \frac{1-k/n}{B(k/n)}};
    \end{equation*}
    \item\label{it:main:3} for all $(s,t) \in \Neigh_2$,
    \begin{equation*}
      \lim_{n \to +\infty} L^{2,n}_{\infty}(s,t) = L_{\infty}(s)L_{\infty}(t).
    \end{equation*}
  \end{enumerate}
\end{theo} 

The point~\eqref{it:main:2} is proved in Section~\ref{s:Laplace}, while the point~\eqref{it:main:3} is detailed in Section~\ref{s:cv}.

As is stated in Corollary~\ref{cor:main} below, Theorem~\ref{theo:main} implies the $P_{\infty}$-chaoticity of $P^n_{\infty}$, and actually yields the convergence of $P^{k,n}_{\infty}$ in a stronger sense than in Definition~\ref{defi:chaoticity}; namely, in Wasserstein distance~\cite{villani}. 

\begin{defi}
  Let $k \geq 1$ and $q \in [1, +\infty)$. The {\em Wasserstein distance} of order $q$ between two probability distributions $\mu$ and $\nu$ on $\R^k$ is defined by
  \begin{equation*}
    \Ws_q(\mu, \nu) := \inf_{(X,Y) \in \Pi(\mu,\nu)} \left(\Exp[|X-Y|^q]\right)^{1/q},
  \end{equation*}
  where $\Pi(\mu,\nu)$ refers to the set of pairs of random variables with marginal distributions $\mu$ and $\nu$.
\end{defi}
\begin{rk}\label{rk:q}
  The definition of $\Ws_q$ depends on the choice of the norm $|\cdot|$ on $\R^k$. But since all norms are equivalent on $\R^k$, all the associated distances $\Ws_q$ are also equivalent. Therefore, convergence results for the $\Ws_q$ topology do not depend on the choice of the underlying norm. In this paper, we take the convention that the Wasserstein distance of order $q$ is defined with respect to the $\ell^q$ norm $|x| := (|x_1|^q + \cdots + |x_k|^q)^{1/q}$ on $\R^k$.
\end{rk}

We derive chaoticity and convergence in Wasserstein distance as a corollary of Theorem~\ref{theo:main}.

\begin{cor}\label{cor:main}
  Under Assumption~\eqref{ass:E},
  \begin{enumerate}[label=(\roman*), ref=\roman*]
    \item\label{it:cor:1} the sequence of stationary distributions $P^n_{\infty}$ of the projected particle system is $P_{\infty}$-chaotic,
    \item\label{it:cor:2} for all $k \geq 1$, for all $q \in [1, +\infty)$,
    \begin{equation*}
      \lim_{n \to +\infty} \Ws_q\left(P^{k,n}_{\infty}, (P_{\infty})^{\otimes k}\right) = 0.
    \end{equation*}
  \end{enumerate}  
\end{cor}

The proof of Corollary~\ref{cor:main} is postponed to Appendix~\ref{s:app}. A summary of the long time and large scale behaviour of the projected particle system is detailed on Figure~\ref{fig:summary}.

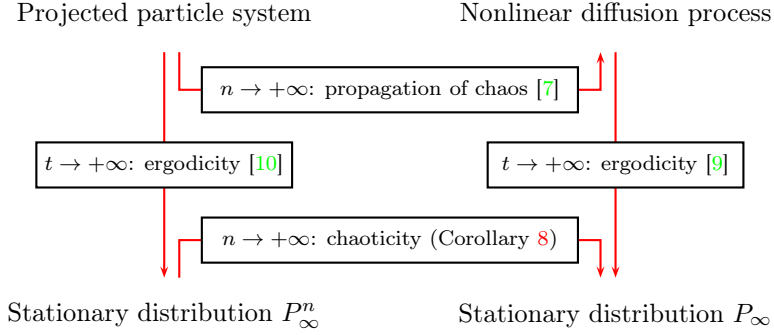
\begin{figure}[ht]
  \begin{pspicture}(10,4)
    \rput(2,4){Projected particle system}
    \rput(2,0){Stationary distribution $P^n_{\infty}$}
    \rput(8,4){Nonlinear diffusion process}
    \rput(8,0){Stationary distribution $P_{\infty}$}
    
    \psline[linecolor=red]{->}(2,3.5)(2,.5)
    \psline[linecolor=red]{->}(8,3.5)(8,.5)
    \psline[linecolor=red]{->}(2.2,3.5)(2.2,3)(7.8,3)(7.8,3.5)
    \psline[linecolor=red]{->}(2.2,.5)(2.2,1)(7.8,1)(7.8,.5)
    
    \pspolygon[fillstyle=solid](2.5,3.3)(2.5,2.7)(7.5,2.7)(7.5,3.3)
    \rput(5,3){\footnotesize $n \to +\infty$: propagation of chaos \cite{jm}}
    \pspolygon[fillstyle=solid](2.5,1.3)(2.5,.7)(7.5,.7)(7.5,1.3)
    \rput(5,1){\footnotesize $n \to +\infty$: chaoticity (Corollary~\ref{cor:main})}
    
    \pspolygon[fillstyle=solid](.3,2.3)(.3,1.7)(3.7,1.7)(3.7,2.3)
    \rput(2,2){\footnotesize $t \to +\infty$: ergodicity~\cite{pp}}
    \pspolygon[fillstyle=solid](6.3,2.3)(6.3,1.7)(9.7,1.7)(9.7,2.3)
    \rput(8,2){\footnotesize $t \to +\infty$: ergodicity~\cite{jourey}}

  \end{pspicture}
  \caption{A summary of convergence results, in long time as well as for a large number of particles, for the projected particle system.}
  \label{fig:summary}
\end{figure}

Figure~\ref{fig:summary} illustrates the fact that, when it makes sense, the interversion of the limits `$n \to +\infty$' and `$t \to +\infty$' is generically correct for functionals of systems of rank-based interacting particles. This remark is of interest in the study of rank-based models of equity markets, such as the Atlas model introduced by Fernholz in the framework of Stochastic Portfolio Theory~\cite{fernholz, banner, ferkar}. Indeed, in this context, relevant quantities such as capital distribution curves or growth rates of portfolios are expressed in terms of the stationary distribution $P^n_{\infty}$ described in Proposition~\ref{prop:pp}. The asymptotic behaviour of these quantities, when the size of the market grows to infinity, where investigated in~\cite{banner, cp}. On the other hand, it was suggested in~\cite{jourey:atlas} to use the propagation of chaos results of~\cite{jourey} to obtain a functional description of an infinite market first, and then apply the available ergodicity results on the nonlinear diffusion process to derive closed formulas for these relevant quantities. Corollary~\ref{cor:main} is a first step toward the validation of the equivalence of both approaches, and we refer to~\cite{jourey:atlas} for a detailed account.

\subsection{Notations} Throughout the paper, we use the following notations: for all $s,t \in \R$, $s \wedge t := \min\{s,t\}$, $s \vee t := \max\{s,t\}$, $[s]^- := 0 \vee (-s)$ and $[s]^+ := 0 \vee s$. Besides, $\lfloor s \rfloor$ denotes the integer part of $s$.


\section{Expression of the Laplace transforms}\label{s:Laplace}

This section is dedicated to the proof the point~\eqref{it:main:2} of Theorem~\ref{theo:main}. We first collect preliminary estimates in Subsection~\ref{ss:prelim}, and then compute the Laplace transform $L^{2,n}_{\infty}(s,t)$ of $P^{2,n}_{\infty}$ in Subsection~\ref{ss:L2ninfty}.

\subsection{Preliminary estimates}\label{ss:prelim} Under Assumption~\eqref{ass:E}, for all $r \in \Neigh$, for all $n \geq 1$, for all $k \in \{1, \ldots, n-1\}$, let us define
\begin{equation*}
  f^+_{k,n}(r) := r \frac{\sigma^2}{2n}\frac{k/n}{B(k/n)}, \qquad f^-_{k,n}(r) := -r \frac{\sigma^2}{2n}\frac{1-k/n}{B(k/n)},
\end{equation*}
so that the quantities $J^n_{i,j}(s,t)$ introduced in Theorem~\ref{theo:main} rewrite
\begin{equation*}
  \begin{aligned}
    & J^n_{i,j}(s,t) := \prod_{k=1}^{i-1} \frac{1}{1 - f^+_{k,n}(s+t)}\prod_{k=i}^{j-1} \frac{1}{1 - f^+_{k,n}(t) - f^-_{k,n}(s)}\prod_{k=j}^{n-1} \frac{1}{1 - f^-_{k,n}(s+t)} \qquad \text{if $i < j$,}\\
    & J^n_{i,j}(s,t) := \prod_{k=1}^{j-1} \frac{1}{1 - f^+_{k,n}(s+t)}\prod_{k=j}^{i-1} \frac{1}{1 - f^+_{k,n}(s) - f^-_{k,n}(t)}\prod_{k=i}^{n-1} \frac{1}{1 - f^-_{k,n}(s+t)} \qquad \text{if $i > j$.}
  \end{aligned}
\end{equation*}

In this subsection, we exhibit upper bounds on the quantities $f^+_{k,n}(r)$ and $f^-_{k,n}(r)$, for $r \in \{s,t,s+t\}$, which ensure that the quantities $J^n_{i,j}(s,t)$ are well defined for $n$ large enough. We roughly proceed as follows: when $k/n$ is far from $1$, then $\frac{k/n}{B(k/n)}$ remains bounded by above, so that $f^+_{k,n}(r)$ is arbitrarily small for $n$ large enough. On the contrary, when $k/n$ is close to $1$, then 
\begin{equation*}
  \frac{k/n}{B(k/n)} \simeq -\frac{1}{b(1)(1-k/n)},
\end{equation*}
so that 
\begin{equation*}
  f^+_{k,n}(r) \simeq -r\frac{\sigma^2}{2b(1)}\frac{1}{n-k},
\end{equation*}
and the fact that $r \in \Neigh$ provides natural bounds on the right-hand side. The same ideas allow to obtain similar bounds on $f^-_{k,n}(r)$.

We now give a rigorous formulation of these arguments. Under Assumption~\eqref{ass:E}, for all $\epsilon > 0$ such that $\epsilon < b(0) \wedge (-b(1))$, we introduce 
\begin{equation*}
  \begin{aligned}
    & \Neigh^{\epsilon} := \{r \in \R : -2(b(0)-\epsilon)/\sigma^2 < r < 2(-b(1)-\epsilon)/\sigma^2\},\\
    & \Neigh^{\epsilon}_2 := \{(s,t) \in \Neigh^{\epsilon} \times \Neigh^{\epsilon} : s+t \in \Neigh^{\epsilon}\}.
  \end{aligned}
\end{equation*}
For all $r \in \Neigh^{\epsilon}$, we define
\begin{equation*}
  \alpha_+(r) := \frac{[r]^+\sigma^2}{2(-b(1)-\epsilon)} \in [0,1), \qquad \alpha_-(r) := \frac{[r]^-\sigma^2}{2(b(0)-\epsilon)} \in [0,1).
\end{equation*}

Let us now fix $(s,t) \in \Neigh^{\epsilon}_2$. Then, for $\delta \in (0,1/2)$ small enough,
\begin{itemize}
  \item for all $u \in [0,\delta]$, $B(u) \geq u(b(0)-\epsilon)$,
  \item for all $u \in [1-\delta,1]$, $B(u) \geq (1-u)(-b(1)-\epsilon)$.
\end{itemize}
Besides, by Assumption~\eqref{ass:E}, we have
\begin{equation*}
  m_-(\delta) := \inf_{u \in [0,1-\delta]} \frac{B(u)}{u} > 0, \qquad m_+(\delta) := \inf_{u \in [\delta,1]} \frac{B(u)}{1-u} > 0.
\end{equation*}

The heuristic arguments detailed at the beginning of the subsection translate into the following precise estimates: for all $r \in \{s,t,s+t\}$, for all $n \geq 1$, for all $k \in \{1, \ldots, n-1\}$,  
\begin{equation}\label{eq:fplus}
  f^+_{k,n}(r) = r \frac{\sigma^2}{2n}\frac{k/n}{B(k/n)} \leq \begin{cases}
    \displaystyle\frac{[r]^+\sigma^2}{2nm_-(\delta)} & \text{if $k < n(1-\delta)$,}\\
    \displaystyle\frac{[r]^+\sigma^2}{2(n-k)}\frac{k/n}{-b(1)-\epsilon} \leq \frac{\alpha_+(r)}{n-k} & \text{if $k \geq n(1-\delta)$.}
  \end{cases}
\end{equation}
Similarly,
\begin{equation}\label{eq:fmoins}
  f^-_{k,n}(r) = -r \frac{\sigma^2}{2n}\frac{1-k/n}{B(k/n)} \leq \begin{cases}
    \displaystyle\frac{[r]^-\sigma^2}{2nm_+(\delta)} & \text{if $k > n\delta$,}\\
    \displaystyle\frac{[r]^-\sigma^2}{2k}\frac{1-k/n}{b(0)-\epsilon} \leq \frac{\alpha_-(r)}{k} & \text{if $k \leq n\delta$.}
  \end{cases}
\end{equation}

In particular, if $n$ is chosen so that
\begin{equation*}
  \frac{[s+t]^+\sigma^2}{2nm_-(\delta)} \leq \frac{1}{2},
\end{equation*}
then we deduce from~\eqref{eq:fplus} that, for all $k \in \{1, \ldots, n-1\}$,
\begin{equation*}
  f^+_{k,n}(s+t) \leq \begin{cases}
    1/2 & \text{if $k < n(1-\delta)$,}\\
    \alpha_+(s+t) & \text{if $k \geq n(1-\delta)$.}
  \end{cases}
\end{equation*}
Similarly, if $n$ is chosen so that
\begin{equation*}
  \frac{[t]^+\sigma^2}{2nm_-(\delta)} + \alpha_-(s) \leq \frac{\alpha_-(s)+1}{2}, \quad \frac{[t]^+\sigma^2}{2nm_-(\delta)} + \frac{[s]^-\sigma^2}{2nm_+(\delta)} \leq \frac{1}{2}, \quad \alpha_+(t) + \frac{[s]^-\sigma^2}{2nm_+(\delta)} \leq \frac{\alpha_+(t)+1}{2},
\end{equation*}
then we deduce from~\eqref{eq:fplus} and~\eqref{eq:fmoins} that, for all $k \in \{1, \ldots, n-1\}$,
\begin{equation*}
  f^+_{k,n}(t)+f^-_{k,n}(s) \leq \begin{cases}
    (\alpha_-(s)+1)/2 & \text{if $k \leq n\delta$,}\\
    1/2 & \text{if $n\delta < k < n(1-\delta)$,}\\
    (\alpha_+(t)+1)/2 & \text{if $k \geq n(1-\delta)$.}
  \end{cases}
\end{equation*}

These results are gathered together in the following lemma.

\begin{lem}\label{lem:prelim}
  Let $(s,t) \in \Neigh_2$. Under Assumption~\eqref{ass:E}, there exists $\epsilon > 0$ such that $(s,t) \in \Neigh_2^{\epsilon}$. Let $\delta \in (0,1/2)$ satisfying the conditions above. Let us define $\bar{\alpha} \in (0,1)$ by
  \begin{equation*}
    \bar{\alpha} := \max\left\{\frac{1}{2}, \alpha_+(s+t), \alpha_-(s+t), \frac{\alpha_+(s)+1}{2}, \frac{\alpha_-(s)+1}{2}, \frac{\alpha_+(t)+1}{2}, \frac{\alpha_-(t)+1}{2}\right\}.
  \end{equation*} 
  Then, there exists $n_0 \geq 2$ such that, for all $n \geq n_0$, for all $k \in \{1, \ldots, n-1\}$, the quantities
  \begin{equation*}
    \begin{aligned}
      & 1 - f^+_{k,n}(s+t), \quad 1 - f^+_{k,n}(t) - f^-_{k,n}(s), \quad 1 - f^+_{k,n}(s) - f^-_{k,n}(t), \quad 1 - f^-_{k,n}(s+t)
    \end{aligned}
  \end{equation*}
  are larger than $1-\bar{\alpha}$.
\end{lem}

\subsection{Computation of \texorpdfstring{$L^{2,n}_{\infty}(s,t)$}{L2n}}\label{ss:L2ninfty} Let us first note that, since $p^n_{\infty}(z)\dd z$ is a symmetric probability distribution on $\R^n$, then for all symmetric and nonnegative function $f : \R^n \to \R$,
\begin{equation*}
  \int_{z \in M_n} f(z_1, \ldots, z_n) p^n_{\infty}(z_1, \ldots, z_n) \dd z = \int_{z \in M_n} f(z_1, \ldots, z_n) \tilde{p}^n_{\infty}(z_1, \ldots, z_n) \dd z,
\end{equation*}
where, for all $z=(z_1, \ldots, z_n) \in M_n$,
\begin{equation*}
  \tilde{p}^n_{\infty}(z_1, \ldots, z_n) = n!\ind{z_1 \leq \cdots \leq z_n} \frac{1}{\mathcal{Z}_n} \exp\left(\frac{2}{\sigma^2} \sum_{k=1}^n b_n(k) z_k\right).
\end{equation*}
Using the symmetry of $p^n_{\infty}(z)\dd z$ again, we deduce that, for all $(s,t) \in \Neigh_2$,
\begin{equation*}
  L^{2,n}_{\infty}(s,t) = \frac{1}{n(n-1)} \sum_{i=1}^n \sum_{j\not=i} \int_{z \in M_n} \exp(sz_i + tz_j) \tilde{p}^n_{\infty}(z)\dd z.
\end{equation*}
Let us now fix $i \in \{1, \ldots, n\}$ and $j \not= i$, and define
\begin{equation*}
  J^n_{i,j}(s,t) := \int_{z \in M_n} \exp(sz_i + tz_j) \tilde{p}^n_{\infty}(z)\dd z.
\end{equation*}
Note that, at this stage, nothing prevents $J^n_{i,j}(s,t)$ from being infinite. Using the parametrisation of $M_n$ by the $n-1$ coordinates $x_1 = z_1, \ldots, x_{i-1} = z_{i-1}, x_{i+1} = z_{i+1}, \ldots, x_n = z_n$, so that the surface measure $\dd z$ on $M_n$ rewrites $\sqrt{n} \dd x_1 \cdots \dd x_{i-1} \dd x_{i+1} \cdots \dd x_n$, we obtain
\begin{equation*}
  \begin{aligned}
    J^n_{i,j}(s,t) & = \frac{n!\sqrt{n}}{\mathcal{Z}_n} \int_{(x_1, \ldots, x_{i-1}, x_{i+1}, \ldots, x_n) \in \R^{n-1}} \ind{x_1 \leq \cdots \leq x_{i-1} \leq -(x_1 + \cdots + x_{i-1} + x_{i+1} + \cdots + x_n) \leq x_{i+1} \leq \cdots \leq x_n}\\
    & \quad \times \exp\left(tx_j + \sum_{k \not=i} \left(\frac{2}{\sigma^2}(b_n(k)-b_n(i))-s\right)x_k\right) \dd x_1 \cdots \dd x_{i-1} \dd x_{i+1} \cdots \dd x_n.
  \end{aligned}
\end{equation*}
We denote $S := x_1 + \cdots + x_{i-1} + x_{i+1} + \cdots + x_n$, and let $y_k := x_k + S$, for all $k \not= i$, in the right-hand side above. Then we have
\begin{equation*}
  \ind{x_1 \leq \cdots \leq x_{i-1} \leq -(x_1 + \cdots + x_{i-1} + x_{i+1} + \cdots + x_n) \leq x_{i+1} \leq \cdots \leq x_n} = \ind{y_1 \leq \cdots \leq y_{i-1} \leq 0 \leq y_{i+1} \leq \cdots \leq y_n},
\end{equation*}
and
\begin{equation*}
  \begin{aligned}
    tx_j + \sum_{k \not=i} \left(\frac{2}{\sigma^2}(b_n(k)-b_n(i))-s\right)x_k & = ty_j + \sum_{k \not=i} \left(\frac{2}{\sigma^2}(b_n(k)-b_n(i))-s\right)y_k\\
    & \quad - S\left(t + \sum_{k \not=i} \left(\frac{2}{\sigma^2}(b_n(k)-b_n(i))-s\right)\right).
  \end{aligned}
\end{equation*}
Note that $S = \sum_{k\not=i} (y_k-S)$, which implies $S = \frac{1}{n} \sum_{k\not=i} y_k$. Besides,
\begin{equation*}
  \sum_{k \not=i} \left(\frac{2}{\sigma^2}(b_n(k)-b_n(i))-s\right) = \sum_{k \not=i} \frac{2}{\sigma^2}b_n(k)-(n-1)\frac{2}{\sigma^2}b_n(i)-(n-1)s = -n\frac{2}{\sigma^2}b_n(i)-(n-1)s,
\end{equation*}
since $\sum_{k=1}^n b_n(k) = n\bar{b} = 0$ by Assumption~\eqref{ass:E}. As a consequence,
\begin{equation*}
  tx_j + \sum_{k \not=i} \left(\frac{2}{\sigma^2}(b_n(k)-b_n(i))-s\right)x_k = \sum_{k\not=i} \gamma^{k,n}_{i,j}(s,t) y_k,
\end{equation*}
where
\begin{equation*}
  \gamma^{k,n}_{i,j}(s,t) := -\frac{s+t}{n} + \frac{2}{\sigma^2}b_n(k) + \ind{k=j} t.
\end{equation*}
As a conclusion,
\begin{equation*}
  \begin{aligned}
    J^n_{i,j}(s,t) & = \frac{n!\sqrt{n}}{n\mathcal{Z}_n} \int_{(y_1, \ldots, y_{i-1}, y_{i+1}, \ldots, y_n) \in \R^{n-1}} \ind{y_1 \leq \cdots \leq y_{i-1} \leq 0 \leq y_{i+1} \leq \cdots \leq y_n}\\
    & \quad \times \exp\left(\sum_{k\not=i} \gamma^{k,n}_{i,j}(s,t) y_k\right) \dd y_1 \cdots \dd y_{i-1} \dd y_{i+1} \cdots \dd y_n\\
    & = \frac{n!\sqrt{n}}{n\mathcal{Z}_n} J^{-,n}_{i,j}(s,t)J^{+,n}_{i,j}(s,t),
  \end{aligned}
\end{equation*}
where
\begin{equation*}
  \begin{aligned}
    & J^{-,n}_{i,j}(s,t) := \int_{y_{i-1}=-\infty}^0 \int_{y_{i-2}=-\infty}^{y_{i-1}} \cdots \int_{y_1=-\infty}^{y_2} \exp(\gamma^{1,n}_{i,j}(s,t) y_1 + \cdots + \gamma^{i-1,n}_{i,j}(s,t) y_{i-1}) \dd y_1 \cdots \dd y_{i-1},\\
    & J^{+,n}_{i,j}(s,t) := \int_{y_{i+1}=0}^{+\infty} \int_{y_{i+2}=y_{i+1}}^{+\infty} \cdots \int_{y_n=y_{n-1}}^{+\infty} \exp(\gamma^{n,n}_{i,j}(s,t) y_n + \cdots + \gamma^{i+1,n}_{i,j}(s,t) y_{i+1}) \dd y_n \cdots \dd y_{i+1}.
  \end{aligned}
\end{equation*}

Let $n_0 \geq 2$ and $\bar{\alpha} \in (0,1)$ be given by Lemma~\ref{lem:prelim}. We deduce from the definition of $\gamma^{k,n}_{i,j}(s,t)$ that, if $n \geq n_0$, then, for all $i,j \in \{1, \ldots, n\}$ such that $i \not= j$, for all $k \in \{1, \ldots, i-1\}$,
\begin{equation*}
  \begin{aligned}
    \gamma^{1,n}_{i,j}(s,t) + \cdots + \gamma^{k,n}_{i,j}(s,t) & = -(s+t)\frac{k}{n} + \frac{2n}{\sigma^2}B\left(\frac{k}{n}\right)\\
    & = \frac{2n}{\sigma^2}B\left(\frac{k}{n}\right)(1-f^+_{k,n}(s+t))\\
    & \geq \frac{2n}{\sigma^2}B\left(\frac{k}{n}\right)(1-\bar{\alpha}) > 0,
  \end{aligned}
\end{equation*}
and similarly, for all $k \in \{i+1, \ldots, n\}$,
\begin{equation*}
  \begin{aligned}
    \gamma^{k,n}_{i,j}(s,t) + \cdots + \gamma^{n,n}_{i,j}(s,t) & = -(s+t)\frac{n-k+1}{n} + t - \frac{2n}{\sigma^2}B\left(\frac{k-1}{n}\right)\\
    & = -\frac{2n}{\sigma^2}B\left(\frac{k-1}{n}\right)(1-f^+_{k-1,n}(t)-f^-_{k-1,n}(s))\\
    & \leq -\frac{2n}{\sigma^2}B\left(\frac{k-1}{n}\right)(1-\bar{\alpha}) < 0,
  \end{aligned}
\end{equation*}
which ensures that $J^{-,n}_{i,j}(s,t)$ and $J^{+,n}_{i,j}(s,t)$ are finite and write
\begin{equation*}
  J^{-,n}_{i,j}(s,t) = \prod_{k=1}^{i-1} \frac{1}{\gamma^{1,n}_{i,j}(s,t) + \cdots + \gamma^{k,n}_{i,j}(s,t)}, \qquad J^{+,n}_{i,j}(s,t) = \prod_{k=i+1}^n \frac{-1}{\gamma^{n,n}_{i,j}(s,t) + \cdots + \gamma^{k,n}_{i,j}(s,t)},
\end{equation*}
thanks to successive integrations. This finally gives
\begin{equation*}
  J^n_{i,j}(s,t) = \frac{n!\sqrt{n}}{n\mathcal{Z}_n} \prod_{k=1}^{i-1} \frac{1}{\frac{-(s+t)k}{n} + \frac{2n}{\sigma^2}B(\frac{k}{n})} \prod_{k=i}^{j-1} \frac{1}{\frac{-tk + s(n-k)}{n} + \frac{2n}{\sigma^2}B(\frac{k}{n})} \prod_{k=j}^{n-1} \frac{1}{\frac{(s+t)(n-k)}{n} + \frac{2n}{\sigma^2}B(\frac{k}{n})}
\end{equation*}
if $i < j$, and
\begin{equation*}
  J^n_{i,j}(s,t) = \frac{n!\sqrt{n}}{n\mathcal{Z}_n} \prod_{k=1}^{j-1} \frac{1}{\frac{-(s+t)k}{n} + \frac{2n}{\sigma^2}B(\frac{k}{n})} \prod_{k=j}^{i-1} \frac{1}{\frac{-sk + t(n-k)}{n} + \frac{2n}{\sigma^2}B(\frac{k}{n})} \prod_{k=i}^{n-1} \frac{1}{\frac{(s+t)(n-k)}{n} + \frac{2n}{\sigma^2}B(\frac{k}{n})}
\end{equation*}
if $i > j$.

To complete the proof, we remark that
\begin{equation*}
  L^{2,n}_{\infty}(0,0) = 1 = \frac{n!\sqrt{n}}{n\mathcal{Z}_n} \prod_{k=1}^{n-1} \frac{1}{\frac{2n}{\sigma^2}B(\frac{k}{n})},
\end{equation*}
which allows us to get rid of the constant term $\frac{n!\sqrt{n}}{n\mathcal{Z}_n}$ and to obtain the expected expression of $J^n_{i,j}(s,t)$ in the point~\eqref{it:main:2} of Theorem~\ref{theo:main}, for $(s,t) \in \Neigh_2$.


\section{Convergence of the Laplace transforms}\label{s:cv}

The proof of the point~\eqref{it:main:3} of Theorem~\ref{theo:main} works in two steps: first, we prove that, for all $t \in \Neigh$, the Laplace transform $L^{1,n}_{\infty}(t) = L^{2,n}_{\infty}(t,0)$ of $P^{1,n}_{\infty}$ converges to $L_{\infty}(t)$. Second, we check that, for $(s,t) \in \Neigh_2$, the difference between $L^{2,n}_{\infty}(s,t)$ and the product $L^{1,n}_{\infty}(s)L^{1,n}_{\infty}(t)$ vanishes. These two steps are addressed in the respective Subsections~\ref{ss:step1} and~\ref{ss:step2}. The preliminary Subsection~\ref{ss:TL} gathers useful elementary results.

\subsection{Elementary inequalities}\label{ss:TL} We shall use the following inequalities, which are elementary consequences of the Taylor-Lagrange inequality.

\begin{enumerate}[label=(TL\arabic*), ref=TL\arabic*]
  \item\label{it:TL:log} For all $\alpha \in [0,1)$, for all $x \in [-\alpha, +\infty)$, $|\log(1+x)-x| \leq \kappa(\alpha)x^2$, where $\kappa(\alpha) := \frac{1}{2(1-\alpha)^2}$.
  \item\label{it:TL:exp} For all $x,y \in \R$, $|\exp(x)-\exp(y)| \leq \exp(y)\left(|x-y|+|R(x-y)|\right)$, where the function $R : z \mapsto \exp(z)-1-z$ is such that, for all $C \in [0,+\infty)$, for all $z \in [-C,C]$, $|R(z)| \leq \frac{1}{2}\exp(C) z^2$.
\end{enumerate}
In particular, we deduce from~\eqref{it:TL:log} that, for all $C \in [0,+\infty)$,
\begin{equation}\label{eq:limsup}
  \lim_{n \to +\infty} \left(1-\frac{C}{n}\right)^{-n} = \exp(C).
\end{equation}

\subsection{Convergence of \texorpdfstring{$L^{1,n}_{\infty}(t)$}{L1n}}\label{ss:step1} Let us fix $t \in \Neigh$. By the results of Section~\ref{s:Laplace}, there exists $n_0 \geq 2$ such that, for all $n \geq n_0$, the Laplace transform $L^{1,n}_{\infty}(t) = L^{2,n}_{\infty}(t,0)$ of $p^{1,n}_{\infty}$ is finite and writes
\begin{equation*}
  L^{1,n}_{\infty}(t) = \frac{1}{n} \sum_{i=1}^n I^n_i(t), \qquad I^n_i(t) := \prod_{k=1}^{i-1} \frac{1}{1-t\frac{\sigma^2}{2n} \frac{k/n}{B(k/n)}} \prod_{k=i}^{n-1} \frac{1}{1+t\frac{\sigma^2}{2n} \frac{1-k/n}{B(k/n)}}.
\end{equation*}
Then, we have
\begin{equation}\label{eq:cv:1}
  |L^{1,n}_{\infty}(t) - L_{\infty}(t)| \leq \sum_{i=1}^n \int_{u=\frac{i-1}{n}}^{\frac{i}{n}} |I^n_i(t) - \exp(t \Phi(u))|\dd u.
\end{equation}

Let $\epsilon > 0$ and $\delta \in (0,1/2)$ be given by Subsection~\ref{ss:prelim} for the pair $(t,0) \in \Neigh_2$. We split the sum appearing in the right-hand side of~\eqref{eq:cv:1} into {\em boundary terms}, corresponding to $i \leq n\delta$ and $i \geq n(1-\delta)$, and a {\em central term}, corresponding to $n\delta < i < n(1-\delta)$. These terms are addressed separately, in the respective~\S\ref{sss:boundary} and~\S\ref{sss:central}. 

\subsubsection{Boundary terms}\label{sss:boundary} For all $n \geq n_0$,
\begin{equation*}
  \sum_{i \leq n\delta} \int_{u=\frac{i-1}{n}}^{\frac{i}{n}} |I^n_i(t) - \exp(t \Phi(u))|\dd u \leq \frac{1}{n}\sum_{i=1}^{\lfloor n\delta \rfloor} I^n_i(t) + \int_{u=0}^{\delta} \exp(t\Phi(u))\dd u.
\end{equation*}
It is an easy consequence of the point~\eqref{it:main:1} in Theorem~\ref{theo:main} that the integral in the right-hand side above vanishes with $\delta$. The purpose of this paragraph is to show that
\begin{equation}\label{eq:cv:2}
  \lim_{\delta \dto 0} \limsup_{n \to +\infty} \frac{1}{n}\sum_{i=1}^{\lfloor n\delta \rfloor} I^n_i(t) = 0.
\end{equation}

Let us first assume that $t \geq 0$. Then, for all $i \leq n\delta$,
\begin{equation*}
  \prod_{k=i}^{n-1} \frac{1}{1+t\frac{\sigma^2}{2n} \frac{1-k/n}{B(k/n)}} \leq 1.
\end{equation*}
We now use the fact that, if $i \leq n\delta$, then for all $k \in \{1, \ldots, i-1\}$, $\frac{k/n}{B(k/n)} \leq \frac{1}{b(0)-\epsilon}$, to write
\begin{equation*}
  \prod_{k=1}^{i-1} \frac{1}{1-t\frac{\sigma^2}{2n} \frac{k/n}{B(k/n)}} \leq \left(1-\frac{t\sigma^2}{2n(b(0)-\epsilon)}\right)^{-(i-1)} \leq \left(1-\frac{t\sigma^2}{2n(b(0)-\epsilon)}\right)^{-n},
\end{equation*}
as soon as $n$ is large enough to ensure that $t\sigma^2 / (2n(b(0)-\epsilon) < 1$. Using~\eqref{eq:limsup}, we deduce that
\begin{equation*}
  \limsup_{n \to +\infty} \frac{1}{n}\sum_{i=1}^{\lfloor n\delta \rfloor} I^n_i(t) \leq \delta \exp\left(\frac{t\sigma^2}{2(b(0)-\epsilon)}\right),
\end{equation*}
and~\eqref{eq:cv:2} easily follows.

Let us now assume that $t < 0$. Then, for $i \leq n\delta$, we still have the rough bound
\begin{equation*}
  \prod_{k=1}^{i-1} \frac{1}{1-t\frac{\sigma^2}{2n} \frac{k/n}{B(k/n)}} \leq 1,
\end{equation*}
but we need to be more careful as far as the product
\begin{equation*}
  \prod_{i=k}^{n-1} \frac{1}{1+t\frac{\sigma^2}{2n} \frac{1-k/n}{B(k/n)}} = \prod_{i=k}^{n-1} \frac{1}{1-f^-_{k,n}(t)}
\end{equation*}
is concerned. On the one hand, 
\begin{equation*}
  \log\left( \prod_{k=\lfloor n\delta \rfloor + 1}^{n-1} \frac{1}{1-f^-_{k,n}(t)}\right) = - \sum_{k=\lfloor n\delta \rfloor + 1}^{n-1} \log\left(1-f^-_{k,n}(t)\right),
\end{equation*} 
and combining Lemma~\ref{lem:prelim} with the inequality~\eqref{it:TL:log} yields
\begin{equation*}
  - \sum_{k=\lfloor n\delta \rfloor + 1}^{n-1} \log\left(1-f^-_{k,n}(t)\right) \leq  \sum_{k=\lfloor n\delta \rfloor + 1}^{n-1} \left(f^-_{k,n}(t) + \kappa(\bar{\alpha}) \left(f^-_{k,n}(t)\right)^2\right).
\end{equation*}
Since the definition of $f^-_{k,n}(t)$ yields
\begin{equation*}
  \sum_{k=\lfloor n\delta \rfloor + 1}^{n-1} f^-_{k,n}(t) = - \frac{t\sigma^2}{2n} \sum_{k=1}^{n-1} \ind{k/n > \delta} \frac{1-k/n}{B(k/n)}
\end{equation*}
while~\eqref{eq:fmoins} implies
\begin{equation*}
  \sum_{k=\lfloor n\delta \rfloor + 1}^{n-1} \left(f^-_{k,n}(t)\right)^2 \leq \frac{1}{n} \left(\frac{t\sigma^2}{2m_+(\delta)}\right)^2,
\end{equation*}
we deduce that
\begin{equation*}
  \lim_{n \to +\infty} \log\left( \prod_{k=\lfloor n\delta \rfloor + 1}^{n-1} \frac{1}{1-f^-_{k,n}(t)}\right) = - \frac{t\sigma^2}{2} \int_{v=\delta}^1 \frac{1-v}{B(v)}\dd v.
\end{equation*}
On the other hand,~\eqref{eq:fmoins} gives
\begin{equation*}
  \prod_{k=i}^{\lfloor n\delta \rfloor} \frac{1}{1-f^-_{k,n}(t)} \leq \prod_{k=i}^{\lfloor n\delta \rfloor} \frac{1}{1-\alpha/k},
\end{equation*}
where $\alpha := \alpha_-(t) < 1$. Using~\eqref{it:TL:log} again, we write
\begin{equation*}
  \begin{aligned}
    \log \left(\prod_{k=i}^{\lfloor n\delta \rfloor} \frac{1}{1-\alpha/k}\right) & = - \sum_{k=i}^{\lfloor n\delta \rfloor} \log\left(1-\frac{\alpha}{k}\right)\\
    & \leq \sum_{k=i}^{\lfloor n\delta \rfloor} \left(\frac{\alpha}{k} + \kappa(\alpha)\frac{\alpha^2}{k^2}\right)\\
    & \leq \alpha \sum_{k=i}^{\lfloor n\delta \rfloor} \frac{1}{k} + \kappa(\alpha)\alpha^2\frac{\pi^2}{6}\\
    & \leq \alpha \left(1 + \log(n\delta) - \log(i)\right) + \kappa(\alpha)\alpha^2\frac{\pi^2}{6},
  \end{aligned}
\end{equation*}
so that
\begin{equation*}
  \prod_{k=i}^{\lfloor n\delta \rfloor} \frac{1}{1-f^-_{k,n}(t)} \leq K(\alpha) \delta^{\alpha} \frac{1}{(i/n)^{\alpha}},
\end{equation*}
where $K(\alpha) := \exp(\alpha + \kappa(\alpha)\alpha^2\pi^2/6)$. Since
\begin{equation*}
  \lim_{n \to +\infty} \frac{1}{n} \sum_{i=1}^{\lfloor n\delta \rfloor} \frac{1}{(i/n)^{\alpha}} = \int_{v=0}^{\delta} \frac{\dd v}{v^{\alpha}} = \frac{\delta^{1-\alpha}}{1-\alpha},
\end{equation*}
we conclude that
\begin{equation*}
  \limsup_{n \to +\infty} \frac{1}{n} \sum_{i=1}^{\lfloor n\delta \rfloor} I^n_i(t) \leq \frac{K(\alpha)}{1-\alpha} \delta \exp\left(- \frac{t\sigma^2}{2} \int_{v=\delta}^1 \frac{1-v}{B(v)}\dd v\right) =: M(\delta).
\end{equation*}
To obtain~\eqref{eq:cv:2}, we now have to check that $M(\delta)$ vanishes with $\delta$. To this aim, we fix $0 < \eta < b(0)\wedge(-b(1))$ such that $t \in \Neigh^{\eta}$. Since the diverging integral $\int_{v=\delta}^1 \frac{1-v}{B(v)}\dd v$ is equivalent to $-\log(\delta)/b(0)$ when $\delta$ vanishes, we deduce that, for $\delta$ small enough, we have
\begin{equation*}
  \int_{v=\delta}^1 \frac{1-v}{B(v)}\dd v \leq \frac{-\log \delta}{b(0)-\eta},
\end{equation*}
so that
\begin{equation*}
  \exp\left(- \frac{t\sigma^2}{2} \int_{v=\delta}^1 \frac{1-v}{B(v)}\dd v\right) \leq \delta^{-\beta}, \qquad \text{with } \beta := \frac{-t\sigma^2}{2(b(0)-\eta)} \in (0,1).
\end{equation*}
As a conclusion, $M(\delta)$ is of order $\delta^{1-\beta}$ when $\delta$ is small, whence~\eqref{eq:cv:2}. 

The boundary term corresponding to $i \geq n(1-\delta)$ can be handled by symmetric arguments.

\subsubsection{Central term}\label{sss:central} We now prove that
\begin{equation}\label{eq:cv:3}
  \lim_{\delta \dto 0} \limsup_{n \to +\infty} \sum_{i = \lfloor n \delta \rfloor + 1}^{\lfloor n (1-\delta) \rfloor - 1} \int_{u=\frac{i-1}{n}}^{\frac{i}{n}} |I^n_i(t)-\exp(t\Phi(u))|\dd u = 0.
\end{equation}
To this aim we fix $i \in \{\lfloor n \delta \rfloor + 1, \ldots, \lfloor n (1-\delta) \rfloor - 1\}$ and $u \in [\frac{i-1}{n}, \frac{i}{n}]$. By~\eqref{it:TL:exp},
\begin{equation*}
  |I^n_i(t)-\exp(t\Phi(u))| \leq \exp(t\Phi(u)) (|\Delta| + |R(\Delta)|),
\end{equation*}
where $\Delta := \Delta_1 + \Delta_2$, with
\begin{equation*}
  \begin{aligned}
    & \Delta_1 := - \sum_{k=1}^{i-1} \log\left(1-\frac{t\sigma^2}{2n}\frac{k/n}{B(k/n)}\right) - \frac{t\sigma^2}{2}\int_{v=0}^u \frac{v}{B(v)}\dd v,\\
    & \Delta_2 := - \sum_{k=i}^{n-1} \log\left(1+\frac{t\sigma^2}{2n}\frac{1-k/n}{B(k/n)}\right) + \frac{t\sigma^2}{2}\int_{v=0}^u \frac{1-v}{B(v)}\dd v.
  \end{aligned}
\end{equation*}
For all $k \in \{1, \ldots, i-1\}$, we deduce from Lemma~\ref{lem:prelim}, the inequality~\eqref{it:TL:log} and the estimate~\eqref{eq:fplus} that
\begin{equation*}
  \begin{aligned}
    |\Delta_1| & \leq \left|\sum_{k=1}^{i-1} \frac{t\sigma^2}{2n}\frac{k/n}{B(k/n)} - \frac{t\sigma^2}{2}\int_{v=0}^u \frac{v}{B(v)}\dd v \right| + \frac{\kappa(\bar{\alpha})}{n}\left(\frac{t\sigma^2}{2m_-(\delta)}\right)^2\\
    & \leq \frac{|t|\sigma^2}{2} \sum_{k=1}^{i-1} \int_{v=\frac{k-1}{n}}^{\frac{k}{n}}\left|\frac{k/n}{B(k/n)} - \frac{v}{B(v)}\right| \dd v + \frac{|t|\sigma^2}{2} \int_{v=\frac{i-1}{n}}^u\frac{v}{B(v)} \dd v + \frac{\kappa(\bar{\alpha})}{n}\left(\frac{t\sigma^2}{2m_-(\delta)}\right)^2.
  \end{aligned}
\end{equation*}
Using the uniform continuity of $v/B(v)$ on $[0,1-\delta]$, we deduce that for $n$ large enough, for all $k \in \{1, \ldots, \lfloor n(1-\delta) \rfloor - 1\}$,
\begin{equation*}
  \forall v \in \left[\frac{k-1}{n}, \frac{k}{n}\right], \qquad \left|\frac{k/n}{B(k/n)} - \frac{v}{B(v)}\right| \leq \delta.
\end{equation*}
As a consequence,
\begin{equation*}
  |\Delta_1| \leq \frac{|t|\sigma^2}{2} \delta + \frac{1}{n} \left(\frac{|t|\sigma^2}{2m_-(\delta)} + \kappa(\bar{\alpha}) \left(\frac{t\sigma^2}{2m_-(\delta)}\right)^2\right) =: M_1(n,\delta),
\end{equation*}
and we note that $M_1(n,\delta)$ does not depend on $i \in \{\lfloor n \delta \rfloor + 1, \ldots, \lfloor n (1-\delta) \rfloor - 1\}$ and satisfies
\begin{equation*}
  \lim_{\delta \dto 0} \limsup_{n \to +\infty} M_1(n,\delta) = 0.
\end{equation*}
We similarly construct $M_2(n,\delta)$ satisfying the same conditions as $M_1(n,\delta)$ and such that $|\Delta_2| \leq M_2(n,\delta)$ for $n$ large enough. As a consequence, for $\delta > 0$ small enough and $n$ large enough, we have $|\Delta_1| + |\Delta_2| \leq 1$, so that~\eqref{it:TL:exp} yields
\begin{equation*}
  |I^n_i(t)-\exp(t\Phi(u))| \leq \exp(t\Phi(u)) (M_1(n,\delta) + M_2(n,\delta) + \frac{\exp(1)}{2}(M_1(n,\delta) + M_2(n,\delta))^2),
\end{equation*}
and finally
\begin{equation*}
  \begin{aligned}
    & \sum_{i = \lfloor n \delta \rfloor + 1}^{\lfloor n (1-\delta) \rfloor - 1} \int_{u=\frac{i-1}{n}}^{\frac{i}{n}} |I^n_i(t)-\exp(t\Phi(u))|\dd u\\
    & \qquad \leq \left(M_1(n,\delta) + M_2(n,\delta) + \frac{\exp(1)}{2}(M_1(n,\delta) + M_2(n,\delta))^2\right) \int_{u=0}^1 \exp(t\Phi(u))\dd u,
  \end{aligned}
\end{equation*}
which completes the proof of~\eqref{eq:cv:3}.

\subsection{Convergence of \texorpdfstring{$L^{2,n}_{\infty}(s,t)-L^{1,n}_{\infty}(s)L^{1,n}_{\infty}(t)$}{L2n - L1nL1n}}\label{ss:step2} Let $n_0 \geq 2$ be given by Lemma~\ref{lem:prelim}. Then for all $n \geq n_0$,
\begin{equation*}
  \begin{aligned}
    |L^{2,n}_{\infty}(s,t) - L^{1,n}_{\infty}(s)L^{1,n}_{\infty}(t)| & = \left|\frac{1}{n(n-1)} \sum_{i=1}^n \sum_{j \not= i} J^n_{i,j}(s,t) - \frac{1}{n^2} \sum_{i=1}^n \sum_{j=1}^n I_i^n(s)I_j^n(t)\right|\\
    & \leq \frac{1}{n(n-1)} \sum_{i=1}^n \sum_{j \not= i} |J^n_{i,j}(s,t)-I^n_i(s)I^n_j(t)|\\
    & \quad + \frac{1}{n(n-1)} \sum_{i=1}^n I^n_i(s)I^n_i(t) + \frac{1}{n-1} L^{1,n}_{\infty}(s)L^{1,n}_{\infty}(t).
  \end{aligned}
\end{equation*}
By the results of Subsection~\ref{ss:step1}, the last term in the right-hand side above vanishes when $n$ grows to infinity. The diagonal term $\frac{1}{n(n-1)} \sum_{i=1}^n I^n_i(s)I^n_i(t)$ is addressed in~\S\ref{sss:diagonal}, and the main term $\frac{1}{n(n-1)} \sum_{i=1}^n \sum_{j \not= i} |J^n_{i,j}(s,t)-I^n_i(s)I^n_j(t)|$ is addressed in~\S\ref{sss:main}.

\subsubsection{Diagonal term}\label{sss:diagonal} In this paragraph, we prove that
\begin{equation}\label{eq:diagonal}
  \lim_{n \to +\infty} \frac{1}{n(n-1)} \sum_{i=1}^n I^n_i(s)I^n_i(t) = 0.
\end{equation}
To this aim, we write, for all $i \in \{1, \ldots, n\}$,
\begin{equation*}
  I^n_i(s)I^n_i(t) = \prod_{k=1}^{i-1} \frac{1}{1 - \frac{s\sigma^2}{2n}\frac{k/n}{B(k/n)}}\frac{1}{1 - \frac{t\sigma^2}{2n}\frac{k/n}{B(k/n)}}\prod_{k=i}^{n-1} \frac{1}{1 + \frac{s\sigma^2}{2n}\frac{1-k/n}{B(k/n)}}\frac{1}{1 + \frac{t\sigma^2}{2n}\frac{1-k/n}{B(k/n)}},
\end{equation*}
and note that if $st \geq 0$, then $I^n_i(s)I^n_i(t) \leq I^n_i(s+t)$, so that~\eqref{eq:diagonal} follows from the results of Subsection~\ref{ss:step1}. On the other hand, if $st < 0$, say $s < 0 < t$, then
\begin{equation*}
  I^n_i(s)I^n_i(t) \leq \prod_{k=1}^{i-1} \frac{1}{1 - \frac{t\sigma^2}{2n}\frac{k/n}{B(k/n)}}\prod_{k=i}^{n-1} \frac{1}{1 + \frac{s\sigma^2}{2n}\frac{1-k/n}{B(k/n)}}.
\end{equation*}
Let us fix $\epsilon > 0$ and $\delta \in (0,1/2)$ as in Subsection~\ref{ss:prelim}. Arguing as in~\S\ref{sss:boundary}, we obtain
\begin{equation*}
  \limsup_{n \to +\infty} \frac{1}{n} \sum_{i=1}^{\lfloor n\delta \rfloor} \prod_{k=1}^{i-1} \frac{1}{1 - \frac{t\sigma^2}{2n}\frac{k/n}{B(k/n)}}\prod_{k=i}^{n-1} \frac{1}{1 + \frac{s\sigma^2}{2n}\frac{1-k/n}{B(k/n)}} \leq \exp\left(\frac{t\sigma^2}{2(b(0)-\epsilon)}\right) M(\delta) < +\infty,
\end{equation*}
and the same arguments apply to the sum for $i \geq n(1-\delta)$. On the other hand, combining the estimates~\eqref{eq:fplus}, \eqref{eq:fmoins} with~\eqref{eq:limsup} yields
\begin{equation*}
  \limsup_{n \to +\infty} \sum_{i=\lfloor n\delta \rfloor + 1}^{\lfloor n(1-\delta) \rfloor - 1} \prod_{k=1}^{i-1} \frac{1}{1 - \frac{t\sigma^2}{2n}\frac{k/n}{B(k/n)}}\prod_{k=i}^{n-1} \frac{1}{1 + \frac{s\sigma^2}{2n}\frac{1-k/n}{B(k/n)}} \leq \exp\left(\frac{t\sigma^2}{2m_-(\delta)} - \frac{s\sigma^2}{2m_+(\delta)}\right).
\end{equation*}
We deduce that
\begin{equation*}
  \limsup_{n \to +\infty} \frac{1}{n}\sum_{i=1}^n I^n_i(s)I^n_i(t) < +\infty,
\end{equation*}
whence~\eqref{eq:diagonal}.

\subsubsection{Main term}\label{sss:main} In this paragraph, we finally check that
\begin{equation}\label{eq:main}
  \lim_{n \to +\infty} \frac{1}{n(n-1)} \sum_{i=1}^n \sum_{j \not= i} |J^n_{i,j}(s,t)-I^n_i(s)I^n_j(t)| = 0.
\end{equation}
By~\eqref{it:TL:exp}, we have, for all $i,j \in \{1, \ldots, n\}$ such that $i \not= j$,
\begin{equation*}
  |J^n_{i,j}(s,t)-I^n_i(s)I^n_j(t)| \leq I^n_i(s)I^n_j(t) \left(|\rho^n_{i,j}(s,t)| + |R(\rho^n_{i,j}(s,t))|\right),
\end{equation*}
where
\begin{equation*}
  \rho^n_{i,j}(s,t) := \log(J^n_{i,j}(s,t))-\log(I^n_i(s)I^n_j(t)).
\end{equation*}

Remark that $\rho^n_{i,j}(s,t)$ writes as a sum, for $k \in \{1, \ldots, n-1\}$, of terms of the form 
\begin{equation*}
  \log(1-f^{\pm}_{k,n}(s)-f^{\pm}_{k,n}(t)) - \log(1-f^{\pm}_{k,n}(s)) - \log(1-f^{\pm}_{k,n}(t)),
\end{equation*}
where Lemma~\ref{lem:prelim} ensures that each term $f^{\pm}_{k,n}(s)$, $f^{\pm}_{k,n}(t)$ and $f^{\pm}_{k,n}(s)+f^{\pm}_{k,n}(t)$ is lower than $\bar{\alpha} < 1$. As a consequence,~\eqref{it:TL:log} yields
\begin{equation*}
  \begin{aligned}
    & |\log(1-f^{\pm}_{k,n}(s)-f^{\pm}_{k,n}(t)) - \log(1-f^{\pm}_{k,n}(s)) - \log(1-f^{\pm}_{k,n}(t))|\\
    & \qquad \leq \kappa(\bar{\alpha})((f^{\pm}_{k,n}(s)+f^{\pm}_{k,n}(t))^2+(f^{\pm}_{k,n}(s))^2+(f^{\pm}_{k,n}(t))^2)\\
    & \qquad \leq 3\kappa(\bar{\alpha})((f^{\pm}_{k,n}(s))^2+(f^{\pm}_{k,n}(t))^2),
  \end{aligned}
\end{equation*}
hence $|\rho^n_{i,j}(s,t)| \leq 3\kappa(\bar{\alpha})\{F^n_i(s) + F^n_j(t)\}$, where
\begin{equation*}
  \begin{aligned}
    F^n_i(s) & := \sum_{k=1}^{i-1} \left(f^-_{k,n}(s)\right)^2 + \sum_{k=i}^{n-1} \left(f^+_{k,n}(s)\right)^2 = \sum_{k=1}^{i-1} \left(\frac{s\sigma^2}{2n}\frac{k/n}{B(k/n)}\right)^2 + \sum_{k=i}^{n-1} \left(\frac{s\sigma^2}{2n}\frac{1-k/n}{B(k/n)}\right)^2.
  \end{aligned}
\end{equation*}
We deduce from the estimates~\eqref{eq:fplus} and~\eqref{eq:fmoins} that
\begin{equation*}
  \begin{aligned}
    F^n_i(s) & \leq \frac{1}{n} \left(\left(\frac{s\sigma^2}{2m_-(\delta)^2}\right)^2 + \left(\frac{s\sigma^2}{2m_+(\delta)^2}\right)^2\right)\\
    & \quad + \frac{\pi^2}{6}\left(\left(\frac{|s|\sigma^2}{2(b(0)-\epsilon)}\right)^2\ind{i \leq n\delta} + \left(\frac{|s|\sigma^2}{2(-b(1)-\epsilon)}\right)^2\ind{i \geq n(1-\delta)}\right),
  \end{aligned}
\end{equation*}
so that that there exist a nonnegative and finite constant $M(\delta)$, that depends on $\delta$, and a nonnegative and finite constant $C$, that does not depend on $\delta$, such that
\begin{equation*}
  |\rho^n_{i,j}(s,t)| \leq \frac{M(\delta)}{n} + C\left(\ind{i \leq n\delta} + \ind{i \geq n(1-\delta)} + \ind{j \leq n\delta} + \ind{j \geq n(1-\delta)}\right).
\end{equation*}
For $n$ large enough, the right-hand side above is lower that $3C$, so that~\eqref{it:TL:exp} yields
\begin{equation*}
  \begin{aligned}
    & |R(\rho^n_{i,j}(s,t))| \leq \frac{\exp(3C)}{2} \left\{\frac{M(\delta)}{n} + C\left(\ind{i \leq n\delta} + \ind{i \geq n(1-\delta)} + \ind{j \leq n\delta} + \ind{j \geq n(1-\delta)}\right)\right\}^2\\
    & \quad \leq \frac{3}{2}\exp(3C) \left\{\frac{M(\delta)^2}{n^2} + C^2\left(\ind{i \leq n\delta} + \ind{i \geq n(1-\delta)}\right)^2 + C^2\left(\ind{j \leq n\delta} + \ind{j \geq n(1-\delta)}\right)^2\right\}\\
    & \quad = \frac{3}{2}\exp(3C) \left\{\frac{M(\delta)^2}{n^2} + C^2\left(\ind{i \leq n\delta} + \ind{i \geq n(1-\delta)} + \ind{j \leq n\delta} + \ind{j \geq n(1-\delta)}\right)\right\}.
  \end{aligned}
\end{equation*}
As a consequence, there exist a nonnegative and finite constant $M'(\delta)$, that depends on $\delta$, and a nonnegative and finite constant $C'$, that does not depend on $\delta$, such that, for $n$ large enough, for all $i \not= j$ in $\{1, \ldots, n\}$,
\begin{equation}\label{eq:braced}
  \begin{aligned}
    & |J^n_{i,j}(s,t) - I^n_i(s)I^n_j(t)|\\
    & \qquad \leq I^n_i(s)I^n_j(t) \left\{\frac{M'(\delta)}{n} + C'\left(\ind{i \leq n\delta} + \ind{i \geq n(1-\delta)} + \ind{j \leq n\delta} + \ind{j \geq n(1-\delta)}\right)\right\}.
  \end{aligned}
\end{equation}

To complete the proof of~\eqref{eq:main}, we now check that
\begin{equation*}
  \lim_{\delta \dto 0} \limsup_{n \to +\infty} \frac{1}{n(n-1)} \sum_{i=1}^n \sum_{j \not= i} I^n_i(s)I^n_j(t) \left\{\cdots\right\} = 0,
\end{equation*}
where $\{\cdots\}$ refers to the braced term in the right-hand side of~\eqref{eq:braced}. Note that, on account of the results of~\S\ref{sss:diagonal}, it is equivalent to show that
\begin{equation*}
  \lim_{\delta \dto 0} \limsup_{n \to +\infty} \frac{1}{n^2} \sum_{i=1}^n \sum_{j=1}^n I^n_i(s)I^n_j(t) \left\{\cdots\right\} = 0.
\end{equation*}
On the one hand,
\begin{equation*}
  \frac{1}{n^2} \sum_{i=1}^n \sum_{j=1}^n I^n_i(s)I^n_j(t) \frac{M'(\delta)}{n} = \frac{M'(\delta)}{n} L^{1,n}_{\infty}(s)L^{1,n}_{\infty}(t)
\end{equation*}
vanishes when $n$ grows to infinity. On the other hand,
\begin{equation*}
  \limsup_{n \to +\infty} \frac{1}{n^2} \sum_{i=1}^n \sum_{j=1}^n I^n_i(s)I^n_j(t) C'\ind{i \leq n\delta} = C'L_{\infty}(t) \limsup_{n \to +\infty}\frac{1}{n} \sum_{i=1}^{\lfloor n\delta\rfloor} I^n_i(s),
\end{equation*}
and it was proved in~\S\ref{sss:boundary} that the last term in the right-hand side vanishes with $\delta$. Addressing the other boundary terms similarly, we obtain~\eqref{eq:main} and thereby complete the proof of the point~\eqref{it:main:3} in Theorem~\ref{theo:main}.


\appendix
\section{Proof of Corollary~\ref{cor:main}}\label{s:app}

 
\begin{proof}[Proof of Corollary~\ref{cor:main}]
  Since $\Neigh_2$ is an open subset of $\R^2$ containing $(0,0)$, the point~\eqref{it:main:3} of Theorem~\ref{theo:main} implies the weak convergence of $P^{2,n}_{\infty}$ to the product measure $(P_{\infty})^{\otimes 2}$. According to the proof of~\cite[Proposition~2.2, p.~177]{sznitman}, this is enough to ensure the $P_{\infty}$-chaoticity of the sequence $(P^n_{\infty})_{n \geq 1}$, which is the point~\eqref{it:cor:1} of Corollary~\ref{cor:main}. 
  
  Let us now address the point~\eqref{it:cor:2} and fix $k \geq 1$, $q \in [1,+\infty)$. Following Remark~\ref{rk:q} and~\cite[Theorem~6.9]{villani}, to prove that $P^{k,n}_{\infty}$ converges to $(P_{\infty})^{\otimes k}$ in Wasserstein distance of order $q$, it suffices to check that 
  \begin{equation*}
    \lim_{n \to +\infty} \Exp[|X_1^n|^q + \cdots + |X_k^n|^q] = \Exp[|X_1|^q + \cdots + |X_k|^q],
  \end{equation*} 
  where, for all $n \geq 1$, $(X_1^n, \ldots, X_n^n)$ is distributed according to $P^{k,n}_{\infty}$, while $(X_1, \ldots, X_n)$ is distributed according to $(P_{\infty})^{\otimes k}$. Using the linearity of the expectation and the symmetry of $P^n_{\infty}$, we deduce that it is enough to check this result for $k=1$. Then we already know that $X_1^n$ converges in distribution to $X_1$, and we now check that the sequence of random variables $(|X_1^n|^q)_{n \geq 1}$ is uniformly integrable, which implies the convergence of $\Exp[|X_1^n|^q]$ and completes the proof.
  
  To check the uniform integrability of the sequence $(|X_1^n|^q)_{n \geq 1}$, we fix $r > q$ and prove that the sequence $(\Exp[|X_1^n|^r])_{n \geq 1}$ is bounded. To this aim, we fix $\rho > 0$ such that $-\rho \in \Neigh$ and $\rho \in \Neigh$. Then, there exists $M \geq 0$ such that, for all $x \in \R$, $|x|^r \leq M + (\exp(-\rho x) + \exp(\rho x))$, so that, for all $n \geq 1$,
  \begin{equation*}
    \Exp[|X_1^n|^r] \leq M + L^{2,n}_{\infty}(-\rho,0) + L^{2,n}_{\infty}(\rho,0),
  \end{equation*}
  and we deduce from the point~\eqref{it:main:3} of Theorem~\ref{theo:main} that the right-hand side converges to a finite value when $n$ grows to infinity, which implies that the left-hand side is uniformly bounded with respect to $n$.
\end{proof}

\subsection*{Acknowledgements} The author is grateful to Benjamin Jourdain for his very helpful comments on this work.



\end{document}